\documentclass[12pt]{elsarticle}
\usepackage{amsmath,amssymb,amsthm}

\usepackage{orcidlink}
\usepackage{bbm}
\usepackage{enumerate}
\usepackage{tikz}
\usepackage{tikz-cd}
\usepackage{hyperref}
\usepackage{geometry}
\theoremstyle{plain}

  \newtheorem{thm}{Theorem}[section]
  \newtheorem{cor}[thm]{Corollary}
  \newtheorem{lem}[thm]{Lemma}
  \newtheorem{prop}[thm]{Proposition}
  \newtheorem{exmp}[thm]{Example}
  \newtheorem{rem}[thm]{Remark}
  
  \newtheorem{defn}[thm]{Definition}

\newcommand{\da}{\downarrow \hspace{-2pt}}

\newcommand{\ua}{\uparrow \hspace{-2pt}}

\newcommand{\UUa}[0]{\makebox[1ex][l]{\lower.15ex
                                 \hbox{$\uparrow$}}\kern-1ex\lower-.15ex
                                 \hbox{$\uparrow$}}

\newcommand{\DDa}[0]{\makebox[1ex][l]{\lower.15ex
                                 \hbox{$\downarrow$}}\kern-1ex\lower-.15ex
                                 \hbox{$\downarrow$}}

\begin{document}
\begin{frontmatter}
\title{The Scott space of lattice of closed subsets with supremum operator as a topological semilattice\tnoteref{t1}}
\tnotetext[t1]{Research supported by NSF of China (Nos. 12471439, 12231007).}

\author[add2]{Yu Chen}
\ead{eidolon_chenyu@126.com}
\author[add1]{Hui Kou}
\ead{kouhui@scu.edu.cn}
\author[add1]{Zhenchao Lyu\orcidlink{0000-0003-2685-1180}\corref{cor}}
\ead{zhenchaolyu@scu.edu.cn}
\cortext[cor]{Corresponding author}
\author[add1]{Weiyu Yang}
\ead{scuyangweiyu@163.com}

\address[add1]{Department of Mathematics, Sichuan University, Chengdu 610064, China}
\address[add2]{Faculty of Science, Kunming University of Science and Technology, Kunming 650500, China}

\begin{abstract}
  We present several equivalent conditions of the continuity of the supremum function from the square of the Scott space of $C(X)$ to itself under mild assumptions, where $C(X)$ denotes the lattice of closed subsets of a $\mathbf{T_0}$ topological space.
 We  also show that a $\mathbf{T_0}$ space is quasicontinuous (quasialgebraic) iff the lattice of its closed subsets  is a quasicontinuous (quasialgebraic) domain by using $n$-approximation. Furthermore, we provide a necessary condition for when a topological space possesses a Scott completion. This allows us to give more examples which do not have Scott completions. 


  \vskip 2mm

\end{abstract}

\end{frontmatter}

\section{Introduction}

Topological semilattice are closely connected to  the theory of continuous lattices in domain theory.  For instance, Lawson showed that compact unital semilattices with small semilattices are exactly continuous lattices equipped with the Lawson topology \cite[Theorem VI-3.4]{CON03}. From the viewpoint of  category topology, Banaschewski introduced the notion of an essential embedding within the category of $\mathbf{T_0}$ spaces and continuous maps \cite{Bana77}, which has a close relation to Scott's injective spaces. He also obtained a characterization of so-called essentially complete $\mathbf{T_0}$ spaces in terms of join filters. Subsequently, Hoffmann provided a more transparent characterization: essentially complete $\mathbf{T_0}$ spaces  are precisely the sober spaces that form complete lattices under their specialization order, satisfying an additional condition relating the supremum operator to the topology \cite{Hoff79}. Powerdomains play an important role in modeling the semantics of nondeterministic programming languages (see \cite[Chapter 6]{AJ94} and \cite[IV-8]{CON03}). The Hoare powerdomain over a topological space can be realized as the set of its nonempty closed subsets ordered by inclusion.
In this paper we study a special kind of essentially complete $\mathbf{T_0}$ space, 
 namely the Scott space of the lattice of closed subsets of a topological space with the supremum operator.

Directed spaces \cite{YK2015} constitute a special class of $\mathbf{T_0}$ topological spaces which extend the concept of Scott spaces. Within $\mathbf{T_0}$ spaces, the notion of directed spaces is equivalent both to monotone determined spaces (as established by Ern\'{e} \cite{ERNE2009}) and to CRP-spaces (introduced by Battenfeld \cite[Definition 6.7]{Batt13}).
$\mathbf{T_0}$ c-spaces (or $n$-continuous spaces) are all monotone determined spaces. Based on this, Feng and Kou \cite{FK2017} introduced another approximation relation on monotone determined spaces, called $d$-approximation, and defined the quasicontinuous monotone determined spaces. They established an equivalence between quasicontinuous spaces and locally hypercompact spaces. In domain theory, there is a close relationship between the continuity of a dcpo $L$ and that of the lattice $\sigma(L)$ of its Scott-open subsets. More surprisingly, the continuity of a dcpo $L$ can be characterized by the continuity of the lattice $\Gamma(L)$ of its Scott-closed subsets (see \cite[Theorem 3.17]{CK2022}). Moreover, this correspondence extends to general $\mathbf{T_0}$ spaces.
We show that a $\mathbf{T_0}$ space $X$ is  quasicontinuous (quasialgebraic)  iff the lattice of its closed subsets $C(X)$ is quasicontinuous (quasialgebraic) by using $n$-approximation. 

It is well known that for completely distributive (resp., distributive hypercontinuous) lattices $L$, the hull-kernel topology on the spectrum $\mathrm{Spec} L$ coincides with the Scott topology \cite{CON03}.
A longstanding open problem posed by Lawson and Mislove asks: for which distributive continuous lattices does the spectrum, equipped with the hull-kernel topology, form precisely a sober locally compact Scott space? (See \cite[Problem 528]{Mill1990}.)  Chen, Kou, and Lyu \cite{CK2022} established a necessary condition for this problem: $\sigma(L^{\mathrm{op}}) = v(L^{\mathrm{op}})$. Since $L^{\mathrm{op}}$ corresponds to the lattice of closed subsets of a topological space, we investigate equivalent conditions for $\sigma(C(X)) = v(C(X))$ in general topological spaces $X$. 
Surprisingly, the continuity of the supremum map $\sup\colon \Sigma C(X) \times \Sigma C(X) \to \Sigma C(X)$ constitutes one such equivalent condition under mild assumptions. For monotone determined spaces, the additional assumption can be dropped. In addition, we provide an example of non-monotone determined space $X$ such that $\eta=\lambda x.{\downarrow}x\colon X\rightarrow\Sigma C(X)$ is continuous.

Monotone convergence spaces play a fundamental role in domain theory. Every $\mathbf{T_0}$ space admits a free monotone convergence completion, known as its D-completion \cite{KeiLaw2009}. We show that D-completion of any monotone determined space is a Scott space. The concept of Scott completion, introduced by Zhang, Shi, and Li \cite{ZHA2021}, refers to a construction that slightly differs from D-completion. As a corollary, we establish that for any monotone determined space, its Scott completion coincides with its D-completion. The Scott completion of monotone determined spaces are shown to be very useful to construct free dcpo algebras over any general dcpo (see \cite{CKL23,CKLX24}).  
But the notion does not seem to have been studied much. We provide a necessary condition for when a space has Scott completion. Based on this we give a simple proof of a previous result in \cite[Theorem 4.1]{Kou01}.

{\it{Outline}}. In Section 2, we recall some useful knowledge about monotone determined spaces. We discuss  D-completion of monotone determined spaces and provide more examples of topological spaces which do not have Scott completions are presented in Section 3.  In Section 4 the Scott space of the lattice of closed subsets of a topological space with the jointly continuous supremum operator is discussed. We characterize the quasicontinuity of lattice of closed subsets in Section 5.
\section{Preliminaries}

We assume some basic knowledge of domain theory and topology, as in, e.g., \cite{AJ94,CON03,GAU}.

A nonempty set $P$ endowed with a partial order $\leq$  is called a \emph{poset}. For $A\subseteq P$, we set $\da A=\{x\in P: \exists a\in A,  \ x\leq a\}$,\ $\ua A =\{x\in P: \exists a\in P, \ a\leq x\}$.\ $A$ is called a \emph{lower} or an \emph{upper} set, if $A=\da A$ or $A=\ua A$ respectively. For an element $a\in P$, we use $\da a$ or $\ua a$ instead of $\da\{a\}$ or $\ua \{a\}$, respectively.

Let $P$ be a poset. We denote $\sigma(P)$, $\upsilon(P)$ and $A(P)$ to be the \emph{Scott topology},\ the \emph{upper topology} and the \emph{Alexandroff topology} on $P$ respectively. 
For a complete lattice $L$, and $x,y \in L$, we say that $x \prec y$ iff $y \in \mathrm{int}_{\upsilon(L)} \ua x$. $L$ is \textit{hypercontinuous} if $\{d \in L:d \prec x\}$ is directed and has $x$ as a supremum for all $x\in L$. $L$ is \textit{hyperalgebraic} if $x = \bigvee \{d\in L: d \prec d \leq x\}$ for all $x \in L$. An element $y$ such that $y \prec y$ is called \textit{hypercompact}.

Topological spaces will always be supposed to be $\mathbf{T_0}$. For a topological space $X$, its topology is denoted by $\mathcal{O}(X)$ or $\tau$. The partial order $\sqsubseteq$ defined on $X$ by $x\sqsubseteq y \Leftrightarrow  x\in \overline{\{y\}}$
is called the  \textit{specialization order}, where $\overline{\{y\}}$ is the closure of $\{y\}$.\ From now on, all order-theoretical statements about  $\mathbf{T_0}$ spaces, such as upper sets, lower sets, directed sets, and so on,  always refer to  the specialization order ``$\sqsubseteq $".\ 
A subset $C$ of a topological space is \textit{irreducible} if it is nonempty and if $C\subseteq A\cup B$, where $A$ and $B$ are closed,
implies that $C\subseteq A$ or $C\subseteq B$. A  $\mathbf{T_0}$ topological space is \textit{sober} if any irreducible closed subset is the closure of a singleton.

For any two topological spaces $X,Y$, we denote $Y^{X}$ or $TOP(X,Y)$ the set of all continuous maps from $X$ to $Y$, endowed with the pointwise order. Denote $[X \to Y]_p$ and $[X \to Y]_I$ to be the topological space equipped with the topology of pointwise convergence and the \textit{Isbell topology} on $Y^{X}$ respectively.

We now introduce the notion of a monotone determined space. 
Let $(X,\mathcal{O}(X))$ be a $\mathbf{T_0}$ space. Every directed subset $D\subseteq X$ can be  regarded as a monotone net $(d)_{d\in D}$. Set
$DS(X)=\{D\subseteq X: D \ {\rm is \ directed}\}$ to be the family of all directed subsets of $X$. For an $x\in X$, we denote  $D\rightarrow x$  to mean  that $x$ is a limit of $D$, i.e., $D$ converges to $x$ with respect to the topology on $X$. Then the following result is obvious.

\begin{lem}  Let $X$ be a $\mathbf{T_0}$ space. For any $(D,x)\in DS(X)\times X$, $D\rightarrow x$ if and only if $D\cap U\not=\emptyset$ for any open neighborhood of $x$.
\end{lem}

Set $DLim(X)=\{(D,x)\in DS(X)\times X: \ D\rightarrow x \}$ to be the set of all pairs of directed subsets and their limits in  $X$. Then $(\{y\},x)\in DLim(X) $ iff $x\sqsubseteq y$ for all $x,y\in X$. 

\begin{defn} Let $X$ be a $\mathbf{T_0}$ space. A subset $U\subseteq X$ is called directed-open  if for all $(D,x)\in DLim(X)$, $x\in U$ implies $D\cap U\not=\emptyset$.
\end{defn}

Obviously, every open set of $X$ is directed-open. Set
$d(\mathcal{O}(X))=\{U\subseteq X: U \ {\rm is} \  {\rm directed}\text{-}{\rm open}\},$
then $\mathcal{O}(X)\subseteq d(\mathcal{O}(X))$.

\begin{thm} {\rm\cite[Theorem 2.4]{YK2015}} \label{convergence}
 Let $X$ be a $\mathbf{T_0}$ topological space. Then
\begin{enumerate} 
\item [{\rm(1)}]  For all $U\in d(\mathcal{O}(X))$, $U=\ua U$;
\item [{\rm(2)}] $X$ equipped with $d(\mathcal{O}(X))$ is a $\mathbf{T_0}$ topological space such that  $\sqsubseteq_d =\sqsubseteq$, where $\sqsubseteq_d$ is the specialization order relative to $d(\mathcal{O}(X))$.
\item [{\rm(3)}] For a directed subset $D$ of $X$, $D\rightarrow x$ iff $D\rightarrow_d x$ for all $x\in X$, where $D\rightarrow_d x$ means that $D$ converges to $x$ with respect to the topology $d(\mathcal{O}(X))$.

\item[{\rm(4)}] $d(d(\mathcal{O}(X))=d(\mathcal{O}(X))$.
\end{enumerate}
\end{thm}

\begin{defn}{\rm\cite[Definition 3.1]{YK2015}} A  topological space $X$ is said to be a monotone determined space if it is $\mathbf{T_0}$ and  every directed-open set is open; equivalently, $d(\mathcal{O}(X))=\mathcal{O}(X)$.
\end{defn}

 Given any space $X$, we denote $\mathcal{D}X$ the topological space $(X,d(\mathcal{O}(X)))$.

\begin{thm}{\rm\cite[Theorem 3.2]{YK2015}} Let $X$ be a $\mathbf{T_0}$ space, we have
\begin{enumerate}
\item[{\rm(1)}] $\mathcal{D}X$ is a monotone determined space.

\item[{\rm(2)}] The following three conditions are equivalent to each other:
\begin{enumerate}
\item[{\rm(i)}] $X$ is a monotone determined space;
\item[{\rm(ii)}] For all $U\subseteq X$, $U$ is open iff for any $(D,x)\in DLim(X)$, $x\in U$ implies $U\cap D\not=\emptyset$.
\item[{\rm(iii)}] For all $A\subseteq X$, $A$ is closed iff for any directed subset $D\subseteq A$, $D\rightarrow x$ implies $x\in A$ for all $x\in X$.
\end{enumerate}
\end{enumerate}
\end{thm}

Monotone determined spaces include many important examples in domain theory.

\begin{exmp}{\rm

\begin{enumerate}[{\rm(1)}]
\item Every poset endowed with the Scott topology is a monotone determined space.

\item (\cite[Section 7]{ERNE2009}) Every poset endowed with the weak Scott topology is a monotone determined space.
\item Every c-space is a monotone determined space. In particular, any poset endowed with the Alexandroff topology is a monotone determined space.
\item (\cite{ERNE2009,FK2017}) Every locally hypercompact space is a monotone determined space.
\end{enumerate}
}
\end{exmp}

We denote ${\bf DTop}$ the category of all nonempty monotone determined spaces with continuous maps as morphisms.

\vskip 3mm

\section{D-completion and Scott completion}
Monotone convergence spaces (a.k.a d-spaces) is a very important class of topological spaces in domain theory. 
It is well known that every $\mathbf{T_0}$ space has a D-completion.
Let us recall some results about D-completion for further discussions.

\begin{defn}
	Let $X$ be a $\mathbf{T_0}$ space. If $X$ with the specialization order is a dcpo and every open set of $X$ is Scott open in $(X,\leq)$, then we call $X$ a monotone convergence space space.
\end{defn} 

A subset $A$ of a poset $P$ is called \textit{d-closed} if for every directed set $D$ of $A$ that possesses supremum $\bigvee D$,  $\bigvee D$ is included in $A$. 
The d-closed sets form the closed sets for a topology, called the \textit{d-topology} \cite[Section 5]{KeiLaw2009}. 
We denote the closure of $A$ in $P$ with d-topology by $\text{cl}_{d}(A)$. If $\text{cl}_{d}(A)$ is equal to $P$, then we say that $A$ is \textit{d-dense} in $P$.  A map $f$ from a poset $P$ to a poset $Q$ is called \textit{d-continuous} if $f$ is continuous under the d-topology. 

\begin{defn}{\rm\cite[Definition 6.5]{KeiLaw2009}}
	An embedding $j\colon X\rightarrow\tilde{X}$ of a space $X$ with image a d-dense subset of a d-space $\tilde{X}$ is called a $D$-completion.
\end{defn}

\begin{lem}{\rm\cite[Lemma 6.3]{KeiLaw2009}}
	Consider the following properties for a subset $A$ of a $\mathbf{T_0}$ space 
	\begin{enumerate}[{\rm(1)}]
		\item $A$ is a monotone convergence space.
		\item $A$ is a sub-dcpo.
		\item $A$ is d-closed.
	\end{enumerate}
	Then $(1)$ implies $(2)$ implies $(3)$, and all three are equivalent if $X$ is a monotone convergence space.
\end{lem}

The following result says that the D-completions are universal and hence they are unique up to isomorphism.

\begin{thm}{\rm\cite[Theorem 6.7]{KeiLaw2009}}\label{D-completion}
	Let $j\colon X\rightarrow Y$ be a topological embedding of a space $X$ into a monotone convergence space $Y$. Let $\tilde{X}=\text{cl}_{d}(j(X))$ be the d-closure of $j(X)$ in $Y$, equipped with the relative topology from $Y$. Then $k\colon X\rightarrow \tilde{X}$, the corestriction of $j$, is a universal $D$-completion, that is, for every continuous map $f$ from $X$ to a monotone convergence space $M$, there is a unique continuous map $\tilde{f}\colon\tilde{X}\rightarrow M$ such that $\tilde{f}\circ k=f$.
	
\end{thm}

\begin{rem}{\rm
For any $\mathbf{T_0}$ space $X$, there is a natural monotone convergence space $C(X)$ (the lattice of closed subsets of $X$) with the upper topology such that $\eta=\lambda x.{\downarrow}x\colon X\rightarrow C(X)$ is an embedding.
We call the d-closure of $\eta(X)$ in $C(X)$ the standard D-completion of $X$ and denote it by $X^d$. }
\end{rem}

\begin{prop}\label{osod}
	Let $X$ be a $\mathbf{T_0}$ space and $X^{d}$ be its D-completion. Then $\mathcal{O}(X)\cong \mathcal{O}(X^{d})$.
\end{prop}

\begin{proof}
	Let $f\colon X\rightarrow X^{s}$ be the soberification map. Since $X^{s}$ is a d-space, the D-completion $X^{d}$ of $X$ can be regarded as a subspace of $X^{s}$ by Theorem \ref{D-completion}, then $X^{s}$ is also a soberification of $X^{d}$. We conclude that 
	$\mathcal{O}(X)\cong \mathcal{O}(X^{d})$ directly.
\end{proof}

\begin{thm}
Let $X$ be a monotone determined space and $X^d$ be the D-completion of $X$. Then $X^d$ is a Scott space, i.e. the topology on $X^d$ equals to the Scott topology on $(X^d,\leq)$.
\end{thm}

\begin{proof}
Since $X^d$ is a monotone convergence space, we only need to show that every Scott open set of $(X^d,\leq)$ is open in $X^d$. For any Scott open set $\mathcal{U}$ of $(X^d,\leq)$ and any closed set $A\in\mathcal{U}$, notice that $$A\in X^d=\text{cl}_d\{{\downarrow}x:x\in X\}=\text{cl}_d\{{\downarrow}x:x\in A\}\cup\text{cl}_d\{{\downarrow}x:x\in X\backslash A\}.$$
Since $\{{\downarrow}x:x\in X\backslash A\}\subseteq\Diamond(X\backslash A)=\{F\in C(X):F\cap(X\backslash A)\neq\emptyset\}$ and $\Diamond(X\backslash A)$ is a upper set hence is a d-closed set in $C(X)$, it follows that $\text{cl}_d\{{\downarrow}x:x\in X\backslash A\}\subseteq\Diamond(X\backslash A)$. Then we have that $A\not\in\text{cl}_d\{{\downarrow}x:x\in X\backslash A\}$.
Hence $A$ is included in $\text{cl}_d\{{\downarrow}x:x\in A\}\subseteq X^d$. Because $\mathcal{U}$ is d-open set of $X^d$ and $\text{cl}_d\{{\downarrow}x:x\in A\}$ is d-closed set, we obtain that $\mathcal{U}\cap\{{\downarrow}x:x\in A\}\neq\emptyset$. By using the fact that $\eta\colon X\rightarrow \Sigma(X^d,\leq)$ is continuous since $X$ is a monotone determined space, we take the inverse image of $\mathcal{U}$ and obtain an open set $U=\eta^{-1}(\mathcal{U})$. Obviously $U\cap A\neq\emptyset$. Hence we have that  $A\in \Diamond U\cap X^d\subseteq\mathcal{U}$. Therefore $\mathcal{U}$ is an open subset of $X^d$.
\end{proof}

\begin{rem}{\rm
The aforementioned theorem also appeared in \cite[Proposition 6.13]{Batt13} with a different proof.}
\end{rem}

The following results are direct consequences of the above theorem.
\begin{cor}
\begin{enumerate}[{\rm(1)}]
\item {\rm\cite[Theorem~7.4]{KeiLaw2009}}
Let $P$ be a poset. The D-completion of $\Sigma P$ is a Scott space. 
\item {\rm\cite[Theorem~4.7]{ZHA2021}}
Let $X$ be a monotone determined space. $(X^d,i)$ is a Scott completion of $X$.
\end{enumerate}	

\end{cor}

The D-completion of any monotone determined space is a Scott space while there exists a non-monotone determined space whose D-completion is a Scott space. This example also can be found in \cite{Batt13}.

\begin{exmp}\label{Non-directed}{\rm
Let $D$ be a set $\mathbb{N}\times\mathbb{N}\cup\{a_n:n\in\mathbb{N}\}\cup\{\top\}$. The order on $D$ is defined as follows:
$x\leq y$ iff
\begin{itemize}
\item $y=\top$;
\item There are natural numbers $n\leq m$ such that $x=a_n,y=a_m$;
\item There are natural numbers $n$ and $m_1\leq m_2$ such that $x=(n,m_1),y=(n,m_2)$;
\item There are natural numbers $n$ and $m$ such that $x=(n,m),y=a_n$. 
\end{itemize}
$D$ can be easily depicted as in Figure 1.
Let $X=\mathbb{N}\times\mathbb{N}\cup\{\top\}$ with the subspace topology of $\Sigma D$. It is easy to see that 
$\Sigma D$ is a D-completion of $X$.  Notice that $\{\top\}$
is always a directed open set of $X$ which is not open in $X$. Hence $X$ is not a monotone determined space.}
\end{exmp}

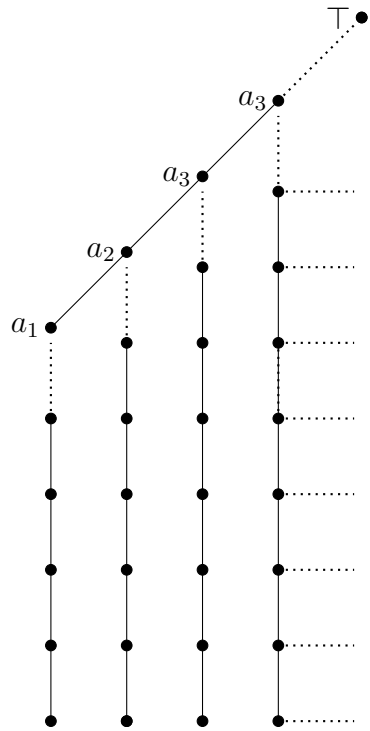
\begin{figure}
	\centering	
	\begin{tikzpicture}
		\draw (-1,0)--(-1,4);
		\draw (0,0)--(0,5);
		\draw (1,0)--(1,6);
		\draw (2,0)--(2,7);
		
		\draw[fill]  (1,0) circle (2pt);	\draw[fill]  (1,1) circle (2pt);	\draw[fill]  (1,2) circle (2pt);	\draw[fill]  (1,3) circle (2pt);	\draw[fill]  (1,4) circle (2pt);	\draw[fill]  (1,5) circle (2pt);	\draw[fill]  (1,6) circle (2pt);	\draw[fill]  (2,0) circle (2pt);
		\draw[fill]  (0,5) circle (2pt);
		\draw[fill] (0,0) circle (2pt); \draw[fill] (0,0) circle (2pt); \draw[fill] (-1,0) circle (2pt);	
		\draw[fill] (0,1) circle (2pt);\draw[fill] (2,1) circle (2pt);\draw[fill] (-1,1) circle (2pt); 
		\draw[fill] (0,2) circle (2pt);\draw[fill] (2,2) circle (2pt);\draw[fill] (-1,2) circle (2pt); 
		\draw[fill] (0,3) circle (2pt);\draw[fill] (2,3) circle (2pt);\draw[fill] (-1,3) circle (2pt); 
		\draw[fill] (0,4) circle (2pt);\draw[fill] (2,4) circle (2pt);\draw[fill] (-1,4) circle (2pt); 
	    \draw[fill] (2,5) circle (2pt);\draw[fill] (2,6) circle (2pt);\draw[fill] (2,7) circle (2pt);
	    \draw[fill] (3.1,9.3) circle (2pt);
	
		\draw[thick,dotted] (-1,4)--(-1,5);\draw[thick,dotted] (2,4)--(2,5);
		\draw[thick,dotted] (0,5)--(0,6);
		\draw[thick,dotted] (1,6)--(1,7);
		\draw[thick,dotted] (2,7)--(2,8);
		\draw[thick,dotted]  (2,8.2)--(3,9.2);
		\draw[thick,dotted] (2,0)--(3,0);
		\draw[thick,dotted] (2,1)--(3,1);
		\draw[thick,dotted] (2,2)--(3,2);
		\draw[thick,dotted] (2,3)--(3,3);
		\draw[thick,dotted] (2,4)--(3,4);
		\draw[thick,dotted] (2,5)--(3,5);
		\draw[thick,dotted] (2,6)--(3,6);
		\draw[thick,dotted] (2,7)--(3,7);

		\draw [fill] (-1,5.2)node[anchor=east]{$a_1$} circle (2pt);\draw (-1,5.2)--(0,6.2)--(1,7.2)--(2,8.2);
		\draw [fill] (0,6.2)node[anchor=east]{$a_2$} circle (2pt);
		\draw [fill] (1,7.2)node[anchor=east]{$a_3$} circle (2pt);
		\draw [fill] (2,8.2)node[anchor=east]{$a_3$} circle (2pt);
		\draw [fill] (3.1,9.3)node[anchor=east]{$\top$} circle (2pt);		
	\end{tikzpicture}
	\caption{$\mathbb{N}\times\mathbb{N}\cup\{a_n:n\in\mathbb{N}\}\cup\{\top\}$}  
\end{figure}

The concept of Scott completion is introduced by Zhang, Shi and Li in \cite{ZHA2021}. It has played an important role in constructing free dcpo algebras over general dcpos, see \cite{CK2022,CKL23}. But the notion does not seem to have been investigated much. We will give more examples of topological spaces which do not have Scott completions in this section.

\begin{defn}{\rm\cite{ZHA2021}}
	A Scott completion $(Y, f)$ of a space $X$ is a Scott space $Y$ together with
	a continuous map $f\colon X\rightarrow Y$ such that for any Scott space $Z$ and continuous map
	$g\colon X\rightarrow Z$, there exists a unique continuous map $\tilde{g}$ satisfying $g =\tilde{g}\circ f$.
	
\end{defn}

\begin{thm}\label{Iso}
	Let $X$ be a $\mathbf{T_0}$ space and $(\tilde{X},i)$ be a Scott-completion of $X$, then $\mathcal{O}(X)\cong\mathcal{O}(\tilde{X})$. Hence $C(X)\cong C(\tilde{X})$.
\end{thm}

\begin{proof}
	Let $[Y\rightarrow 2]$ denote the set of continuous functions from $Y$ to the Sierpinski space $2$ with the pointwise order. It is well known that $\mathcal{O}(Y)\cong [Y\rightarrow 2]$. For any $f\in [X\rightarrow 2]$, there exists only one $\bar{f}\in [\tilde{X}\rightarrow 2]$ such that $f=\tilde{f}\circ i$ since $(\tilde{X},i)$ is a Scott-completion of $X$. Hence we obtain  a map $\theta\colon[X\rightarrow 2]\rightarrow[\tilde{X}\rightarrow 2]$ as follows
	$$\theta(f)=\bar{f}.$$
	Clearly $\theta$ is bijective.
	
	$\theta$ is monotone. For any $f_1,f_2\in [X\rightarrow 2]$ with $f_1\leq f_2$. We need to show that $\bar{f_1}\leq\bar{f_2}$. Without loss of generality, we consider $y\in\bar{X}$ with $\bar{f_1}(y)=1$, i.e., $y\in \bar{f_1}^{-1}(1)$. We claim that $i(X)\cap{\downarrow}y\cap\bar{f_1}^{-1}(1)\neq\emptyset$.
	Otherwise we have the indicator function $\mathbbm{1}_{U}$ of $U=({\downarrow}y)^c\cap\bar{f_1}^{-1}(1)$ with $\mathbbm{1}_{U}\circ i=\bar{f}_1\circ i=f$. It follows that $\mathbbm{1}_{U}=\bar{f}_1$. Then $U=\mathbbm{1}_{U}^{-1}(1)=\bar{f}^{-1}_1(1)$. This is impossible. Back to $i(X)\cap{\downarrow}y\cap\bar{f_1}^{-1}(1)\neq\emptyset$.
	We have some $x\in X$ with $i(x)\in {\downarrow}y\cap\bar{f_1}^{-1}(1)$. It means that $f_1(x)=\bar{f}_1(i(x))=1$. Since $f_1\leq f_2$, we obtain that $f_2(x)=1$. Then $\bar{f}_2(y)\geq\bar{f}_2(i(x))=f_2(x)=1$.
	
	It is easy to see that $\theta$ reflects order. Hence $\theta$ is an order isomorphism. So we obtain that $\mathcal{O}(X)\cong\mathcal{O}(\tilde{X})$.
\end{proof}

\begin{rem}\label{emb}{\rm
		The above proof is similar to the argument of the order isomorphism between the open lattice of a topological space
		and the open lattice of its soberification.
		The map $i^{-1}\colon\mathcal{O}(\tilde{X})\rightarrow\mathcal{O}(X)$
		is an order isomorphic map. Furthermore, $i\colon X\rightarrow\tilde{X}$ is an embedding since $i$ is injective, continuous and almost open.}
\end{rem}

\begin{cor}\label{nScott}
	\begin{enumerate}[{\rm(1)}]
		\item Every  $T_1$ space endowed with non-discrete topology does not have a Scott completion.
		\item Every sober but non-Scott space does not have a Scott-completion.
	\end{enumerate}
	
\end{cor}

\begin{proof}
	(1) Assume that there is a $T_1$ space $X$ with non-discrete topology such that it has a Scott completion $(\Sigma D,i)$, where $D$ is a dcpo. By Theorem \ref{Iso}, $\mathcal{O}(X)\cong\sigma(D)$. It follows that $C(X)\cong\Gamma(D)$. Notice that any singleton set is a minimal element in $C(X)\backslash\{\emptyset\}$. Since every minimal nonempty Scott closed of $D$ must be the form of $\{x\}$ ($x\in \text{Min}D$), then $D={\uparrow}\text{Min}D$. It is easy to see that $D$ has the discrete order. Hence $\Gamma(D)$ is a completely distributive lattice. It follows that $X$ is a c-space. Hence $X$ is a discrete space. Contradiction. 
	
	(2) Assume that there is a sober space $X$ that is not a Scott space and it has a Scott-completion $(\tilde{X},i)$. Then $i(X)\subsetneqq\tilde{X}$.
	Consider the soberification map $r\colon \tilde{X}\rightarrow\tilde{X}^s$, then $r\circ i\colon X\rightarrow\tilde{X}^s$ is also a soberification map. This is impossible.
\end{proof}

The following result is a direct application of  Theorem $\ref{Iso}$.
\begin{cor}{\rm\cite[Theorem 4.1]{Kou01}}
	The category of continuous dcpos with Scott continuous maps is not reflective in the category of quasicontinuous dcpos with Scott continuous maps. 
\end{cor}

\begin{proof}
	Notice that $(2,\leq)$ is a continuous domain. The Sierpinski space is $(2,\leq)$ with its Scott topology. Take any quasicontinuous dcpo $D$ which is not continuous. Assume that there exists a continuous dcpo $\tilde{D}$ with a Scott continuous map $i\colon D\rightarrow\tilde{D}$, such that it satisfies the universal property. By the similar argument in Theorem \ref{Iso}, we have $\sigma(D)\cong\sigma(\tilde{D})$. This is impossible.
\end{proof}

\begin{prop}
	Let $X$ be a $\mathbf{T_0}$ space such that it has a Scott-completion $(\tilde{X},i)$.
	Then $\eta\colon X\rightarrow \Sigma C(X)$ is continuous.
\end{prop}

\begin{proof}
	Notice that $j=\lambda x.{\downarrow}x\colon\tilde{X}\rightarrow\Sigma\Gamma(\tilde{X})$ is continuous since $\tilde{X}$ is a Scott space. $i^{-1}\colon\Gamma(\tilde{X})\rightarrow C(X)$ is an order isomorphism by Remark \ref{emb}. We claim that $$\eta=i^{-1}\circ j\circ i.$$
	For any $x\in X$, $i^{-1}\circ j\circ i(x)=i^{-1}({\downarrow}i(x))\supseteq{\downarrow}x=\eta(x)$.
	On other hand, for any $y\in i^{-1}({\downarrow}i(x))$, i.e., $i(y)\leq i(x)$. Since $i\colon X\rightarrow\tilde{X}$ is an embedding by Remark \ref{emb}, $i$ reflects order. It follows that $y\leq x$. Therefore $\eta$ is continuous.
\end{proof}

\begin{cor}
	Let $X$ be a $\mathbf{T_0}$ topological space such that its  D-completion is a Scott space. Then $\eta\colon X\rightarrow \Sigma C(X)$ is continuous.
\end{cor}

\begin{exmp}{\rm
		There exists a topological space $X$ such that $\eta\colon X\rightarrow \Sigma C(X)$ is continuous while $X$ does not have a Scott completion. Let $\mathcal{J}$ be Johnstone's non-sober dcpo \cite[Exercise 5.2.15]{GAU}. Notice that ${\displaystyle Id(\prod^{n}\mathcal{J})=\prod^{n}Id(\mathcal{J})}$, 
		by using Corollary \ref{product} and Proposition \ref{ideal}
		we obtain that $\sigma(\Gamma (\mathcal{J}))=v(\Gamma (\mathcal{J}))$. Now let $X$ be the soberification of $\Sigma \mathcal{J}$. It implies that $\sigma(C (X))=v(C(X))$. Since $X$ is not a Scott space, $X$ does not have a Scott completion by Corollary \ref{nScott}. It is easy to verify that $\eta\colon X\rightarrow \Sigma C(X)$ is continuous.}
\end{exmp}

\section{Continuity of the supremum function $\Sigma C(X)\times\Sigma C(X)\rightarrow\Sigma C(X)$}

It is well known that the spectrum with hull-kernel topology of a completely distributive lattice (resp. a distributive hypercontinuous lattice) is exactly a continuous (resp. quasicontinuous) dcpo endowed with the Scott topology \cite{LAW81,Hoff81,Lawson1979}. There is a long-standing open problem asked by Lawson and Mislove that, for which distributive continuous lattice its spectrum is exactly a sober locally compact Scott space (see \cite[Problem 528]{Mill1990}).
 Chen, Kou and Lyu \cite{CK2022} found a necessary condition  for this problem, which is $\sigma(L^{op})=v(L^{op})$. Since $L^{op}$ can be viewed as a lattice of closed subsets of a topological space, we continue to investigate the equivalent conditions of $\sigma(C(X))=v(C(X))$ for general topological spaces $X$.
Surprisingly the continuity of the supremum function $\Sigma C(X)\times\Sigma C(X)\rightarrow\Sigma C(X)$ is one of these equivalent conditions under mild assumptions.

Given any poset $P$, $\upsilon(\Gamma (P)) = \sigma(\Gamma(P))$ is a necessary condition for $\Sigma P$ to be core compact \cite{CK2022}. 
Given a topological space $X$, we denote $\prod\limits^n X$ to be topological product of $n$ copies of $X$. For any $n\in \mathbb{N}$, define a map $s_n : \prod\limits^n X \to  \Sigma(C(X))$ as follows: $\forall (x_1,x_2,\dots,x_n) \in \prod\limits^n X$,
$$s_n(x_1,x_2,\dots,x_n) =\ \da\! {\{x_1,x_2,\dots,x_n\}}.$$

The following result is due to R.Hoffmann.
\begin{thm}{\rm\cite[Proposition 1.1]{Hoff79}}\label{ecs}
A space is an essentially complete $\mathbf{T_0}$ space iff the following are satisfied:
\begin{enumerate}[{\rm(i)}]
    \item  The associated pre-order (i.e., specialization order) $(X,\leq)$ is a complete lattice.
    \item The topology on $X$ is coarser than the Scott topology of $(X,\leq)$.
    \item Whenever $a=\sup\{b,c\}$ with $a,b,c\in X$, then $\{U\cap V | b\in U\in\mathcal{O}(X), c\in U\in\mathcal{O}(X)\}$ is an open neighborhood basis for $a$ in $X$.
\end{enumerate}
\end{thm}
Hoffmann noticed that the above condition $(iii)$ can be replaced by that ``$\sup\colon X\times X\rightarrow X$ is continuous'' \cite[Lemma 1.5]{Hoff79}.


\begin{thm}\label{opop}
For a topological space $X$, the following conditions are equivalent:
\begin{enumerate}[{\rm(1)}]
\item $\sigma(C(X)) = \upsilon(C(X))$;
\item $s_n$ is continuous for all $n \in \mathbb{N}$;
\item $s_1$ is continuous and $\bigcup\colon\Sigma C(X)\times\Sigma C(X)\rightarrow\Sigma C(X)$ is continuous.
\end{enumerate}

\end{thm}

\begin{proof} 
$(1)\implies (3)$. Notice that $$\Diamond U=\{A\in C(X):A\cap U\neq\emptyset\}, U\in\mathcal{O}(X)$$ form a subbasis of  the upper topology on $C(X)$. It is a direct argument.

$(3)\implies (2)$. From $$X\times X\stackrel{s_1\times s_1}{\longrightarrow}\Sigma C(X)\times\Sigma C(X)\stackrel{\bigcup}{\longrightarrow}\Sigma C(X),$$ we obtain that $s_2=\bigcup\circ(s_1\times s_1)$. It follows that $s_2$ is continuous. By induction on $n$, it is easy to see that $s_n$ is continuous.

$(2)\implies (1)$. For the completeness, we give the detailed proof here.
 Assume that $s_n$ is continuous for all $n \in \mathbb{N}$. Let $\mathcal{U}$ be an open subset of $\Sigma(C(X))$ and $A\in\mathcal{U}$. Without loss of generality, we assume $A\not=\emptyset$. Note that since $A=\bigcup\{\da\! F: F\subseteq_f\!A\}$ and $\{\da F: F\subseteq_fA\}$ is  a directed family in $C(X)$, there exists a non-empty finite subset $F$ of $ A$ such that $\da \! F\in \mathcal{U}$. Let $F=\{x_1,x_2,\ldots,x_n\}$, then $s_n(x_1,x_2,\ldots,x_n)=\da\! F\in \mathcal{U}$. It follows that  $(x_1,x_2,\ldots,x_n)\in s_n^{-1}(\mathcal{U})$. By the continuity of $s_n$, there exists a family of open subsets $U_k (1 \leq k \leq n)$
such that $U_1\times U_2\times\cdots\times U_n$ is open in $\prod\limits^n X$ and 
$$(x_1,x_2,\ldots,x_n)\in U_1\times U_2\times\cdots\times U_n\subseteq s_n^{-1}(\mathcal{U}).$$
Since $x_k\in A$ for $1\leq k\leq n$, we have $A\in \Diamond U_k=\{B\in C(X): \ B\cap U_k\not=\emptyset\}$. It follows that $A\in \bigcap\limits_{k=1}^n\Diamond U_k\in \upsilon(C(X))$.
For any $B\in \bigcap\limits_{k=1}^n\Diamond U_k$, there exists $y_k\in B\cap U_k$ for $1\leq k\leq n$. Since $(y_1,y_2,\dots,y_n)\in s_n^{-1}(\mathcal{U})$, we have $\bigcup\limits_{k=1}^n\da y_k\in \mathcal{U}$. It follows that $B\in \mathcal{U}$, i.e., $A\in\bigcap\limits_{k=1}^n\Diamond U_k\subseteq \mathcal{U}$.
\end{proof}

\begin{rem}{\rm
The equivalence of $(1)$ and $(2)$ has already been given in \cite[Proposition 4.1]{CK22}.}
\end{rem}

\begin{figure}
	\centering	
	\begin{tikzpicture}
		\draw (1,-1)--(-1,0)--(-1,5);
		\draw (1,-1)--(0,0)--(0,5);
		\draw (1,-1)--(1,0)--(1,5);
		\draw (1,-1)--(2,0)--(2,5);
		\draw[fill]  (1,0) circle (2pt);
        \draw[fill]  (1,1) circle (2pt);	\draw[fill]  (1,2) circle (2pt);	\draw[fill]  (1,3) circle (2pt);	\draw[fill]  (1,4) circle (2pt);	\draw[fill]  (1,5) circle (2pt);		\draw[fill]  (2,0) circle (2pt);
		\draw[fill]  (0,5) circle (2pt);
		\draw[fill] (0,0) circle (2pt); \draw[fill] (0,0) circle (2pt); \draw[fill] (-1,0) circle (2pt);	
		\draw[fill] (0,1) circle (2pt);\draw[fill] (2,1) circle (2pt);\draw[fill] (-1,1) circle (2pt); 
		\draw[fill] (0,2) circle (2pt);\draw[fill] (2,2) circle (2pt);\draw[fill] (-1,2) circle (2pt); 
		\draw[fill] (0,3) circle (2pt);\draw[fill] (2,3) circle (2pt);\draw[fill] (-1,3) circle (2pt); 
		\draw[fill] (0,4) circle (2pt);\draw[fill] (2,4) circle (2pt);\draw[fill] (-1,4) circle (2pt);
        \draw[fill] (-1,5) circle (2pt);
	    \draw[fill] (2,5) circle (2pt);
      \draw[thick,dotted] (1,-1)--(2.5,0);
        \draw[thick,dotted] (1,-1)--(2.8,0);
        \draw[thick,dotted] (2.2,-0.5)--(2.7,-0.5);
		\draw[thick,dotted] (-1,5)--(-1,6);\draw[thick,dotted] (2,5)--(2,6);
		\draw[thick,dotted] (0,5)--(0,6);
		\draw[thick,dotted] (1,5)--(1,6);
		\draw[thick,dotted] (2,0)--(3,0);
		\draw[thick,dotted] (2,1)--(3,1);
		\draw[thick,dotted] (2,2)--(3,2);
		\draw[thick,dotted] (2,3)--(3,3);
		\draw[thick,dotted] (2,4)--(3,4);
		\draw[thick,dotted] (2,5)--(3,5);
		\draw [fill] (1,7)node[anchor=south]{$\top$} circle (2pt);
		\draw [fill] (1,-1)node[anchor=north]{$\bot$} circle (2pt);
	\end{tikzpicture}
	\caption{P}  
\end{figure}

\begin{exmp}\label{countable chain}{\rm
There exists a space $X$ such that $\bigcup\colon\Sigma C(X)\times\Sigma C(X)\rightarrow\Sigma C(X)$ is continuous while $s_1\colon X\rightarrow\Sigma C(X)$ is not continuous. We consider the example $P$ with its upper topology from \cite[Example 4.9]{CK2022}, let $X=(P,v(P))$.
It is easy to see that $C(X)=\{{\downarrow}F:F\subseteq_{fin}P\}\cup\{P,\emptyset\}$. Next we verify that  $\bigcup\colon\Sigma C(X)\times\Sigma C(X)\rightarrow\Sigma C(X)$ is continuous. For any Scott open subset $\mathcal{U}$ of $C(X)$, 
$\bigcup^{-1}(\mathcal{U})=\{(A,B)\in C(X)\times C(X):A\cup B\in\mathcal{U}\}$. For any $A\cup B\in\mathcal{U}$, we construct two subsets of $C(X), \mathcal{V}=\mathcal{U}\cup\{A'\in C(X):A\subseteq A'\}$, $\mathcal{W}=\mathcal{U}\cup\{B'\in C(X):B\subseteq B'\}$. We claim that both $\mathcal{V}$ and $\mathcal{W}$ are Scott open subsets of $C(X)$. Obviously $\mathcal{V}$ is upward closed. For any directed family $\{F_i\}_{i\in I}$ consisting of closed subsets of $X$ such that $\sup^{\uparrow} F_i\in\mathcal{V}$, if $\sup^{\uparrow} F_i=P\in\mathcal{V}$, then there exists some $F_i\in\mathcal{U}\subseteq\mathcal{V}$ since $\mathcal{U}$ is Scott open. If $\sup^{\uparrow} F_i\neq P$, then there is a maximal element $F_{i_0}$ of $\{F_i:i\in I\}$. Hence $\mathcal{V}$ is Scott open. It can be deduced that $\mathcal{W}$ is Scott open as well. Notice that $(A,B)\in\mathcal{V}\times\mathcal{W}\subseteq\bigcup^{-1}(\mathcal{U})$. Therefore $\bigcup\colon\Sigma C(X)\times\Sigma C(X)\rightarrow\Sigma C(X)$ is continuous. Since $\sigma(C(X))\neq v(C(X))$ from \cite[Example 4.9]{CK2022}, we conclude that $s_1\colon X\rightarrow\Sigma C(X)$ is not continuous by Theorem \ref{opop}.}
\end{exmp}

\begin{cor}
Let $X$ be a monotone determined space. The following statements are equivalent:
\begin{enumerate}[{\rm(1)}]
	\item  $\sigma(C(X)) = \upsilon(C(X))$;
	\item  $s_n$ is continuous for all $n \in \mathbb{N}$;
	\item $\bigcup\colon\Sigma C(X)\times\Sigma C(X)\rightarrow\Sigma C(X)$ is continuous.
\end{enumerate}
\end{cor}

\begin{proof}
Notice that $s_1\colon X\rightarrow\Sigma C(X)$ is continuous since $X$ is a monotone determined space.
\end{proof}

\begin{rem}{\rm
Let $X$ be a monotone determined space. We consider a topology $\tau$ on $C(X)$ with $v(C(X))\subseteq\tau\subseteq\sigma(C(X))$. 
By a similar argument in Theorem \ref{opop}, we can show that $\tau$ is equal to $v(C(X))$ if the operator $\bigcup\colon(C(X),\tau)\times(C(X),\tau)\rightarrow(C(X),\tau)$ is continuous.  In other words, there exists only one topology $\tau$ on $C(X)$ 
such that $(1)$ the specialization order of $\tau$ is equal to the inclusion order and $(2)$
$(C(X),\tau)$ is an essentially complete $\mathbf{T_0}$ space. But this is false for a non-monotone determined space $X$, for instance see Example \ref{countable chain}. }

\end{rem}

\begin{cor}
Let $D$ be a poset. The following conditions are equivalent:
\begin{enumerate}[{\rm(1)}]
	\item  $\sigma(\Gamma(D)) = \upsilon(\Gamma(D))$;
	\item $s_n$ is continuous for all $n \in \mathbb{N}$;
	\item $\bigcup\colon\Sigma \Gamma(D)\times\Sigma \Gamma(D)\rightarrow\Sigma \Gamma(D)$ is continuous.
\end{enumerate}
\end{cor}

The following result is a direct consequence of the above corollary. 
\begin{cor}{\rm\cite[Proposition 3.5]{CK2022}}\label{product}
Let $P$ be a poset. If ${\displaystyle\Sigma(\prod^{n} P)=\prod^{n}(\Sigma P)}$ for each $n\in\mathbb{N}$, then 
$\sigma(\Gamma(P))=v(\Gamma(P))$.
\end{cor}

\vskip 2mm

\begin{prop}
Let $X$ be a $T_1$ space. The following conditions are equivalent:
\begin{enumerate}[{\rm(1)}]
\item $X$ has the discrete topology;
\item $\sigma(C(X))=\upsilon(C(X))$;
\item $s_1\colon X\rightarrow\Sigma C(X)$ is continuous.
\end{enumerate}

\end{prop}

\begin{proof}
$(1)\implies (2)$. Since $C(X)$ is a completely distributive lattice, $\sigma(C(X))=\upsilon(C(X))$.

$(2)\implies(3)$. Straightforward.

$(3)\implies (1)$. Notice that $\mathcal{A}=\{\{a\}:a\in A\}\cup\{\emptyset\}$ is always a Scott closed set in $C(X)$ for any $A\subseteq X$. Take the inverse image of $\mathcal{A}$, we obtain that $s_1^{-1}(\mathcal{A})=A$ is closed in $X$ since $\eta$ is continuous. Hence $X$ has the discrete topology.
\end{proof}

\begin{rem}{\rm
$(2)$ implies $(1)$  can be found in \cite[Proposition 3.15]{CK2022}.The above results do not hold for $\mathbf{T_0}$ spaces. For example we consider $X=\Sigma L$ where $L$ is the Isbell's lattice, then $(3)$ holds while $(2)$ does not hold since $\Sigma C(X)$ is not sober \cite{MLZ2021}.}
\end{rem}

\vskip 2mm

 A natural question arises that whether a topological space $X$ that makes the map $s_1\colon X\rightarrow\Sigma(C(X))$  continuous is a monotone determined space? When $L$ is a monotone convergence space such that $(L,\leq)$ is a complete lattice, the answer is positive (see \cite[Proposition 4.11]{CK22}). But this is not true for general topological spaces. First we recall a useful result about the product topology of Scott spaces.

\vskip 3mm

\begin{prop}{\rm\cite[Lemma 4.1]{MLXZ23}}\label{ideal}
Let $P,Q$ be two posets. If $|\text{Id}(P)|$, $|\text{Id}(Q)|$ are both countable, then 
$\Sigma(P\times Q)=\Sigma P\times\Sigma Q$.
($\text{Id}(P)$ means the set of all non-principal ideals of $P$.)
\end{prop}

\begin{exmp}{\rm
There exists a non-monotone determined space $X$ such that $\eta\colon X\rightarrow\Sigma C(X)$ is continuous and $\sigma(C(X))=v(C(X))$.  We consider the space $X$ in Example \ref{Non-directed}, it follows that $C(X)\cong \Gamma (D)$ since $\Sigma D$ is the D-completion of $X$.
Notice that ${\displaystyle Id(\prod^{n}D)=\prod^{n}Id(D)}$, 
by using Corollary \ref{product} and Proposition \ref{ideal}
we obtain that $\sigma(\Gamma (D))=v(\Gamma (D))$. Hence we also have $\sigma(C(X))=v(C(X))$. It is easy to see that $\eta\colon X\rightarrow\Sigma(C(X))$ is continuous.}

\end{exmp}

Next, we consider the lattice of Scott closed subsets of complete lattices.
The following lemma is an equivalent description of \cite[Proposition 4.10]{CK2022}.
\begin{lem}\label{joinc}
Let $L,M$ be complete lattices and $S=L\times M$. If $~\bigvee\colon\Sigma S\times\Sigma S\rightarrow\Sigma S$ is continuous, then $\Sigma (L\times M)=\Sigma L\times\Sigma M$.
\end{lem}
By the above lemma, if $\Sigma L\times\Sigma M$ is not a Scott space, then $~\bigvee\colon\Sigma S\times\Sigma S\rightarrow\Sigma S$ is not continuous.
We recall that a topological space $X$ is called a {\it{Scott space}} if its topology is equal to the Scott topology on $(X,\leq)$, where $\leq$ is the specialization order of $X$. 

\begin{prop}{\rm\cite[Proposition II-3.15(iii)]{CON03}}\label{Scor}
If $X$ is a continuous retract of a monotone convergence space $Y$, then $X$ must be $\mathbf{T_0}$ and is a monotone convergence space. In the special case that $Y$ is a dcpo equipped with the Scott topology, then $X$ must also have the Scott topology.  
\end{prop}

\begin{thm}
Let $L$ be a complete lattice. Consider the following conditions
\begin{enumerate}[{\rm(1)}]
\item $\bigcup\colon\Sigma\Gamma(L)\times\Sigma\Gamma(L)\rightarrow\Sigma\Gamma(L)$ is continuous;
\item $\bigvee\colon\Sigma L\times\Sigma L\rightarrow\Sigma L$ is continuous.
\end{enumerate}
Then $(1)\implies (2)$, but the converse direction does not hold generally.

\end{thm}

\begin{proof}
$(1)\implies(2)$.\\
Notice that both $\eta=\lambda x.{\downarrow}x\colon L\rightarrow\Gamma(L)$ and $\theta=\lambda A.\sup A\colon\Gamma(L)\rightarrow L$ are Scott continuous. Clearly
$\theta\circ\eta=1_{L}$. Hence $\Sigma L$ is a continuous retraction of $\Sigma\Gamma(L)$. It is direct to verify that 
$\bigvee=\sup\circ\bigcup\circ(\eta\times\eta)\colon\Sigma L\times\Sigma L\rightarrow\Sigma L$. So $\bigvee\colon\Sigma L\times\Sigma L\rightarrow\Sigma L$ is continuous.

We use Hertling's example \cite[Theorem 7.4]{Hert22} to show that there exists a complete lattice $Z$ such that condition $(2)$ holds but condition $(1)$ fails. $Z$ is the disjoint union of a complete lattice $X$ and $Y=\sigma(X)$ plus a top element and a bottom element, i.e., $Z=\{X+Y\}\bigcup\{\top,\bot\}$. $Z$ has the following conditions:
\begin{itemize}
\item $\Sigma X\times\Sigma Y\neq\Sigma(X\times Y)$, $\Sigma Z\times\Sigma Z\neq\Sigma(Z\times Z)$.
\item $\bigvee\colon\Sigma Z\times\Sigma Z\rightarrow\Sigma Z$ is continuous.
\end{itemize}
Assume that $\bigcup\colon\Sigma\Gamma(Z)\times\Sigma\Gamma(Z)\rightarrow\Sigma\Gamma(Z)$ is continuous.
 It is easy to see that 
$\Gamma(Z)\cong\Gamma(X)\times\Gamma(Y)\bigcup\{\top,\bot\}$, the isomorphic map $\theta\colon\Gamma(Z)\rightarrow\Gamma(X)\times\Gamma(Y)\bigcup\{\top,\bot\}$ is defined as follows: $\theta(Z)=\top, \theta(\emptyset)=\bot, \theta(A)=(A\cap X,A\cap Y)$ if $A$ is a proper nonempty Scott closed subset of $Z$.

Clearly $\Sigma(\Gamma(X)\times\Gamma(Y))$ is a subspace of $\Sigma\Gamma(Z)$. Hence $\bigvee\colon\Sigma(\Gamma(X)\times\Gamma(Y))\times\Sigma(\Gamma(X)\times\Gamma(Y))\rightarrow\Sigma(\Gamma(X)\times\Gamma(Y))$ is continuous. By Lemma \ref{joinc}, $\Sigma(\Gamma(X)\times\Gamma(Y))=\Sigma\Gamma(X)\times\Sigma\Gamma(Y)$. Notice that $\Sigma X\times\Sigma Y$ is a retraction of $\Sigma\Gamma(X)\times\Sigma\Gamma(Y)$ which is a Scott space. It follows that $\Sigma X\times\Sigma Y$ is also a Scott space by Proposition \ref{Scor}, i.e., $\Sigma X\times\Sigma Y=\Sigma (X\times Y)$. A contradiction. Consequently we also have 
$\sigma(\Gamma(Z))\neq v(\Gamma(Z))$.

\end{proof}

\section{Quasicontinuity of the lattice of closed subsets}

 We introduce the concepts of quasicontinuous spaces and quasialgebraic spaces. Then we will summarize the correspondences among the continuity of a $\mathbf{T_0}$ space, the continuity of the lattice of its open subsets, and the continuity of the lattice of its closed subsets.

Given any set $X$, we use $F \subseteq_{f} X$ to denote that $F$ is a finite subset of $X$.\ Given any two subsets $G,H \subseteq X$,\ we define $G \leq H$ iff $\ua H \subseteq \ua G$.\ A family of finite sets is said to be directed if given $F_{1}, F_{2}$ in the family,\ there exists an $F$ in the family such that $F_{1},F_{2} \leq F$.\
 Let $X$ be a $\mathbf{T_0}$ topological space and $\mathcal{M}$ be a directed family of finite subsets of $X$.\ Then $F_{\mathcal{M}} = \{A \subseteq X : \exists M \in \mathcal{M}, \ua M \subseteq A\}$ is a filter under the inclusion order.\ We say that $\mathcal{M}$ converges to $x$ if $F_{\mathcal{M}}$ converges to $x$,\ i.e.,\ for any open neighbourhood $U$ of $x$,\ there exists some $A \in \mathcal{M}$ such that $A \subseteq U$.\

\begin{defn}
Let $X$ be a $\mathbf{T_0}$ space and $G,H \subseteq X$.\ We say that $G$ $n$-$approximates$ $H$, denoted by $G \ll_{n} H$, if for every net $(x_{j})_{J}$ of $X$, $(x_{j})_{J} \to y$ for some $y\in H$ implies $(x_{j})_{J} \cap \ua G \neq \emptyset$. We write $G \ll_{n} x$ for $G \ll_{n} \{ x \}$. $G $ is said to be $n$-compact if $G$ is a finite set and $G \ll_{n} G $.
\end{defn}

\begin{lem} \label{n-quasicontinuity}
For any $\mathbf{T_0}$ space $X$, the following are equivalent: for all $x,y\in Y$,
\begin{enumerate}[{\rm(1)}]
\item $G\ll_n H$;
\item  $H \subseteq (\ua G)^{\circ}$, where  $(\ua G)^{\circ}$ means the interior of $\ua G$.
\end{enumerate}
\end{lem}

{\noindent\bf Proof.} $(2)\Rightarrow (1)$ is obvious. We only show $(1)\Rightarrow (2)$. Assume that $y\in H, y \not\in (\ua G)^{\circ}$. Then, for any open set $U\subseteq X$ with $y\in U$, there exists a $y_U\in U$ such that $y_U\not\in \ua G$. Set $J=\{U\in \mathcal{O}(X): y\in U\}$ ordered as: $\forall U,V\in J$,$U\leq V\Leftrightarrow V\subseteq U.$
Then $J$ is directed.  Easily, one can see that  $(y_U)_{U\in J}$ is a net converging to $y$. Since $G \ll_n y$, there exists $U_0\in J$ such that $x \sqsubseteq y_{U_0}$ for some $x \in G$ by definition. It is a contradiction for $y_{U_0}\not\in \ua G$.  Hence, $y \in (\ua G)^{\circ}$. \ $\Box$

\vskip 3mm
 Since the $n$-approximation does not depend on monotone determined space, a notion of $n$-quasicontinuity for a $\mathbf{T_0}$ space can be defined as follows.

\begin{defn}
A $\mathbf{T_0}$ topological space $X$ is called \textit{$n$-quasicontinuous} if for each $x\in X$, there is a directed family $D(x) \subseteq fin_{n}(x) = \{F: F\ \text{is finite}, F \ll_n x\}$ such that $D(x) \rightarrow x$.
\end{defn}

\begin{defn}
A $\mathbf{T_0}$ topological space $X$ is called \textit{locally hypercompact} (or \textit{finitary compact}), if for any open subsets $U$ and $x \in U$, there exists some $F \subseteq_{f}  X$ such that $x \in (\ua F)^{\circ} \subseteq \ua F \subseteq U$. $X$ is called \textit{hypercompactly based} if for any open subsets $U$ and $x \in U$, there exists some $F \subseteq_{f}  X$ such that $x \in (\ua F)^{\circ} = \ua F \subseteq U$.
\end{defn}

Since $G \ll_n H$ iff $H \subseteq (\ua G)^\circ$, definitions of $n$-quasicontinuous and locally hypercompact are equivalent. 
Similarly, we can define  $n$-quasialgebraic spaces. And the following statements hold. 

\begin{thm}  {\rm\cite{FK2017,LAW85}} \label{n-quasicontinuous space}
For a $\mathbf{T_0}$ topological space $X$, the following are equivalent:
\begin{enumerate}[{\rm(1)}]
\item $X$ is  $n$-quasicontinuous;
\item $X$ is a locally hypercompact space;
\item $\mathcal{O}(X)$ is hypercontinuous.

\end{enumerate}
\end{thm}

\vskip 2mm

\begin{defn}
A $\mathbf{T_0}$ topological space is called \textit{quasicontinuous} if one of the four equivalent conditions in Theorem \ref{n-quasicontinuous space} holds; called \textit{quasialgebraic} if it is $n$-quasialgebraic.
\end{defn}

Next, we show that for a monotone determined space $X$, $\mathcal{O}(X)$ is algebraic iff it is hyperalgebraic.

\begin{lem}{\rm\cite[Lemma 2.3]{CK2022}} \label{hyperbelow}
Let $X$ be a topological space and $U,V \in \mathcal{O}(X)$. $U \prec V$ iff $U \subseteq \ua F \subseteq V$ for some $F \subseteq_{f} X$.
\end{lem}

\begin{thm}{\rm\cite[Proposition 7]{ERNE2009}  \label{com=hypercom}}
Let $X$ be a monotone determined space. $\mathcal{O}(X)$ is algebraic iff it is hyperalgebraic.
\end{thm}

\begin{rem} \rm
Lemma \ref{com=hypercom} does not always hold for general $\mathbf{T_0}$ spaces. A simple example is $\mathbb{N}$, the set of natural numbers,  endowed with the cofinite topology. Any open subset of it is compact, but not hypercompact.
\end{rem}

By the above discussion, we conclude the relationships between the continuity of a $\mathbf{T_0}$ space and the lattice of its open subsets as follows.

\begin{thm}\label{open lattice}
Let $X$ be a $\mathbf{T_0}$ space. 
\begin{enumerate}[{\rm(1)}]
\item $X$ is continuous (algebraic) iff $\mathcal{O}(X)$ is a completely distributive lattice (completely distributive algebraic lattice).
\item $X$ is quasicontinuous iff the lattice $\mathcal{O}(X)$ is a hypercontinuous lattice.
\item $X$ is quasialgebraic iff $\mathcal{O}(X)$ is a hyperalgebraic lattice. Moreover, if $X$ is a monotone determined space, then $X$ is quasialgebraic iff $\mathcal{O}(X)$ is an algebraic lattice.

\end{enumerate}
\end{thm}

Given any $\mathbf{T_0}$ space $X$, denote $C(X)$ the lattice of its closed subsets. We will study the relations between $X$ and  $C(X)$. We first recall some useful conditions for continuity and distributiveness of lattices.

\begin{lem}{\rm \cite[Theorem I-3.16]{CON03}} \label{completely distributive}
Let $L$ be a complete lattice. The following statements are equivalent:
\begin{enumerate}[{\rm(1)}]
\item $L$ is completely distributive;
\item $L^{\rm op}$ is completely distributive;
\item $L$ is continuous and every element is the sup of co-primes.
\end{enumerate}
\end{lem}

\begin{lem}{\rm\cite{CON03,LAW81}} \label{hyper op} \label{dis hyper}
The following statements hold.
\begin{enumerate}[{\rm (1)}]

\item Let $L$ be a distributive continuous (algebraic) lattice, then $L$ is hypercontinuous (hyperalgebraic) iff $L^{\rm op}$ is hypercontinuous (quasialgebraic).

\item A complete lattice $L$ is a quasicontinuous (quasialgebraic) lattice iff $\omega(L)$ is a continuous (algebraic) lattice.
\end{enumerate}

\end{lem}

\begin{prop}
Let $X$ be a $\mathbf{T_0}$ space. If $C(X)$ is a continuous lattice,  then $X$ is continuous.
\end{prop}

\begin{proof}
$C(X)$ is a completely distributive lattice by Lemma \ref{completely distributive} since it has enough co-primes. It follows that $\mathcal{O}(X)$ is completely distributive. Hence $X$ is continuous by Theorem \ref{open lattice}(1).
\end{proof}

The above result can be generalized to quasicontinuity. Whether an analogous result holds for core-compactness, however, is still unknown.

\begin{prop}{\rm \label{quasi op}}
Let $X$ be a $\mathbf{T_0}$ space. If $C(X)$ is a quasicontinuous lattice, then $X$ is quasicontinuous.
\end{prop}
\begin{proof}

Let $C(X)$ be a quasicontinuous domain. Then for any $x \in X$, there exists some $\{A_1,\dots,A_n\} \ll\ \da x$ in $C(X)$. 
Given any $a_i \in A_i\ (1 \leq i \leq n)$, we have $\{\da a_1,\dots,\da a_n\} \ll\ \da x$.

{\bf Claim 1.} $\{a_1,\dots,a_n\} \ll_n x$ in $X$.

Given any net $(x_i)_I \to x$, $x \in  \overline{\bigcup_{i \in I} \{x_i\} }$. Then $\da x \subseteq \ \overline{\bigcup_{i \in I} \{x_i\} }$. Let $\mathcal{F} = \{\bigcup_{i \in F}\! \da x_i : F \subseteq_f I \}$. Then $\mathcal{F}$ is directed and $ \da x \leq \bigvee^\ua \mathcal{F} = \overline{\bigcup_{i \in I} \{x_i\} } $ in $C(X)$.  There exist some $F\subseteq_f I$ and $a_j$ such that $\da a_j \subseteq \bigcup_{i\in F} \da x_i$. Thus, there exists a $i \in F$ such that $a_j \leq x_i$, i.e., $\{a_1,\dots,a_n\} \ll_n x$ in $X$.

{\bf Claim 2.} For any $x$ and any open neighbourhood $U$ of $x$, there exists some $F \subseteq_f U$ such that $x \in (\ua F)^\circ$.

 Given any directed family $\mathcal{A}=\{A_i\}_I$ of closed subsets of $X$ such that $\da\! x \leq \bigvee \mathcal{A} = \overline{\bigcup_{i \in I} A_i}$, then $U \cap (\bigcup_{i \in I} A_i) \not = \emptyset$. There exists some $i \in I$ such that $A_i \cap U \not = \emptyset$. Pick $b \in A_i \cap U$, then $\da\! b \subseteq A_i$. We have $\{\da\! b: b \in U\} \ll\ \da\! x $ in $C(X)$. Since $C(X)$ is quasicontinuous, there exists a $\{B_1,\dots,B_m\} \subseteq C(X)$ such that $\{\da b : b \in U \} \ll \{B_1,\dots,B_m\} \ll\ \da x $ in $C(X)$. Then  
 $$\forall 1 \leq i \leq m,\ \{\da b: b \in U\} \ll B_i = \overline{\bigcup_{F\subseteq_f B_i} \da F},$$ 
 which implies that there exists some $F \subseteq_f B_i$ and $b_i \in U$ such that ${\da} b_i \subseteq {\da} F \subseteq B_i$. Thus, $b_i \in B_i \cap U$. It follows that $\{{\da} b_1,\dots,{\da} b_m\} \ll\ {\da} x$ in $C(X)$. Then $\{b_1,\dots,b_n\} \ll_n x$ by Claim 1, i.e., $x \in ({\ua}\{b_1,\dots,b_n\})^\circ$. 
\end{proof}


\begin{thm}
Let $X$ be a $\mathbf{T_0}$ space.
\begin{enumerate}[{\rm(1)}]
\item $X$ is quasicontinuous iff the lattice $C(X)$ is a quasicontinuous lattice.
\item $X$ is quasialgebraic iff the lattice $C(X)$ is a quasialgebraic lattice.
\end{enumerate}
\end{thm}

\begin{proof}

(1) By Proposition \ref{quasi op}, if $C(X)$ is quasicontinuous, then $X$ is quasicontinuous. Conversely, let $X$ be quasicontinuous. Then $\mathcal{O} (X)$ is hypercontinuous by Theorem \ref{open lattice}. From Theorem 
\ref{hyper op}, we know that $C(X)$ is quasicontinuous.

(2) Let $C(X)$ be a quasialgebraic lattice. Then $\omega (C(X)) = \upsilon( \mathcal{O}(X))$ is algebraic by Lemma \ref{hyper op} and $\mathcal{O}(X)$ is hypercontinuous by (2) and Theorem \ref{open lattice}. Then $\sigma(\mathcal{O}(X)) = \upsilon(\mathcal{O}(X))$ is algebraic and completely distributive. Thus, $\mathcal{O}(X)$ is algebraic and then hyperalgebraic. Conversely, let $X$ be a quasialgebraic space. Then $\mathcal{O}(X)$ is distributive hyperalgebraic lattice. $C(X)$ is a quasialgebraic lattice by Theorem \ref{dis hyper}.
\end{proof}

Finally we collect some equivalent characterizations of quasicontinuous space as follows.
\begin{cor}
	Let $X$ be a $\mathbf{T_0}$ topological space. Then the following conditions are equivalent:
	\begin{enumerate}[{\rm(1)}]
		\item  $\mathcal{O}(X)$ with the upper topology is corecompact;
		\item {\rm (\cite[Theorem 2(2)]{LAW85})} $\mathcal{O}(X)$ with the upper topology is an injective space;
		\item $C(X)$ is a quasicontinuous domain;
		\item {\rm (\cite[Theorem VII-3.10]{CON03})} $C(X)$ with Lawson topology is Hausdorff;
		\item $X$ is a quasicontinuous space.
	\end{enumerate}
\end{cor}


\section{Open questions}

\begin{enumerate}[(1)]
\item  Is there a Scott completion $(\tilde{X},i)$ of a topological space $X$, but $(\tilde{X},i)$ is not a D-completion of $X$?
\item  Under the conditions of Theorem \ref{opop}, does $\Sigma C(X)\times\Sigma C(X)=\Sigma(C(X)\times C(X))=\Sigma C(X\sqcup X)$ hold?
\item
Let $X$ be a $\mathbf{T_0}$ space and $\Sigma C(X)$ be core-compact. Is $X$  core-compact?
\end{enumerate}

\vskip 3mm

\noindent{\bf Acknowledgements}
\vskip 3mm
We thank Yuxu Chen and Peixian Zhou for helpful discussions. We also would like to thank one anonymous referee who points out that one of our main results is not new.

\vskip 3mm

\bibliographystyle{plain}
\bibliography{Scott}

\end{document}